\documentclass[a4paper,10pt]{amsart}
\usepackage[T1]{fontenc}
\usepackage[utf8]{inputenc}
\usepackage{amsmath,amsthm,amsfonts,amssymb,latexsym}
\usepackage{amsmath}
\usepackage{mathrsfs}

\def\H2{{\mathbb{H}^{2}}}
\def\bbs{{S^{\ast}}}
\def\bbS{{S}}
\def\C{{\mathbb{C}}}
\theoremstyle{definition}
\theoremstyle{remark}
\numberwithin{equation}{section}

\newtheorem{theorem}{\textbf{Theorem}}[section]
\newtheorem{corollary}[theorem]{\textbf{Corollary}}
\newtheorem{lemma}[theorem]{\textbf{Lemma}}

\newtheorem{proposition}[theorem]{\textbf{Proposition}}

\newtheorem{definition}{Definition}[section]
\newtheorem{remark}{\textbf{Remark}}[section]

\begin{document}
\title{A sharpened Schwarz-Pick operatorial inequality for nilpotent operators}
\author{Haykel GAAYA}
\address{\ddag.Institute Camille Jordan, Office 107 University of Lyon1,
43 Bd November 11, 1918, 69622-Villeurbanne, France.}
\email{\ddag  gaaya@math.univ-lyon1.fr }
\maketitle
\begin{abstract}
Let denote by $S(\phi)$ the extremal operator defined by the compression of the unilateral shift $\bbS$ to the model subspace 
 $ H(\phi)=\H2\ominus\phi~\H2 $ as the following
$\bbS(\phi)f(z)=P(zf(z)),$ where $P$ denotes the orthogonal projection from $\H2$ onto $ H(\phi)$ and $\phi$ is an inner function 
on the unit disc. In this mathematical notes, we give an explicit formula of the numerical radius of $\bbS(\phi)$ in the particular case
 where $\phi$ is a finite Blaschke product with unique zero  and an estimate on the general case. We establish also a sharpened 
Schwarz-Pick operatorial inequality generalizing a U. Haagerup and P. de la Harpe result for nilpotent operators \cite{Haagerup}.
\end{abstract}
\section{Introduction}
\hspace{0.5cm}Let $\mathcal{H}$ be a complex separable Hilbert space and $\mathscr{B}(\mathcal{H})$ the collection of all bounded 
linear operators on $\mathcal{H}$. The numerical range of an operators $T$ in $\mathscr{B}(\mathcal{H})$ is the subset 
$$W(T)=\left\lbrace <Tx,x>\in\C;x\in \mathcal{H} ,\lVert x \lVert= 1\right\rbrace $$  of the plane, where $<.,.>$ denotes 
the inner product in $\mathcal{H}$ and the numerical radius of $T$ is defined by 
$$ \omega_{2}(T)=\sup \left\lbrace \lvert z\lvert; z\in W(T) \right\rbrace .$$
$\mathscr{R}e(T)$ is the self-adjoint operator defined by $$\mathscr{R}e(T)=\frac{1}{2}(T+T^{\ast}).$$
 We denote by $\bbS$ the unilateral shift acting on the Hardy space $\H2$ of the square summable analytic functions and 
by $\bbs$ its adjoint:
$$\begin{array}{ccccc}
\bbS & : & \H2 & \to & \H2 \\
& & f & \mapsto & zf(z) \\
\end{array}$$
$$\begin{array}{ccccc}
\bbs & : & \H2 & \to & \H2 \\
& & f & \mapsto & \dfrac{f(z)-f(0)}{z} \\
\end{array}.$$

\hspace{0.5cm}Beurling's theorem implies that the non zero invariant subspaces of $\bbS$ are of the forme $\phi~\H2$, 
where $\phi$ is some inner function . Let $\bbS(\phi)$ denote the compression of $\bbS$ to the model subspace
 $ H(\phi)=\H2\ominus\phi~\H2 $ defined by:
$$\bbS(\phi)f(z)=P(zf(z)),$$ where $P$ denotes the orthogonal projection from $\H2$ onto $ H(\phi)$.
 We denote by $\bbs(\phi)$ the adjoint of $\bbS(\phi)$: 
$$\bbs(\phi)=\bbS(\phi)^{\ast}={\bbs}_{\lvert H(\phi)}=\bbs_{\lvert Ker(\phi(\bbS)^{\ast})} ~.$$ 

\hspace{0.5cm}The model operator $\bbS(\phi)$ has many properties (See \cite{Nikolski} and \cite{Bercovici}) and it was studied intensively in the 1960s and '70s. 
For example, it has norm 1 (for dim  $H(\phi)>1$) and it is cyclic. The function $\phi$ is the minimal function 
of $\bbS(\phi)$ meaning that $\phi(\bbS(\phi))=0$ and $\phi$ divides any function $\psi$ in $H^{\infty}$ with $\psi(\bbS(\phi))=0$. 
The space $H(\phi)$ is  finite-dimensional exactly when $\phi$ is a finite Blaschke product: 
$$\phi(z)=\prod_{j=1}^{n}\dfrac{z-\alpha_{j}}{1-\overline{\alpha_{j}}z}.$$
\hspace{0.5cm}In this case the polynomial $p(z)=\prod_{j=1}^{n}(z-\alpha_{j})$ is both the minimal and characteristic polynomial 
of $\bbS(\phi)$ and $(\alpha_{j})_{1\leqslant j\leqslant n}$ are its eigenvalues. In particular, if $\phi(z)=z^{n}$ then
 $\bbS(\phi)$ is unitarily equivalent to $\bbS_{n}$ where 

$$
\bbS_{n}=\left(
\begin{array}{cccc}
0 &       &       &  \\
1 & \ddots &      &   \\
  & \ddots &\ddots    &  \\
  &        & 1    & 0 
  
\end{array}
\right)
.$$

\hspace{0.5cm}Note that if $\phi$ is a finite Blaschke product, $\bbS(\phi)$ is in a special class of operator introduced in \cite{Wu2} 
by H. -L. Gau and P. Y. Wu that we will denote by $\Upsilon_{n}$ and which consists
of all completely nonunitary contractions $T$ on a $n$-dimensional space ($\|T\|\leq1$ and $T$ has non eigenvalue of modulus 1) 
with a rank $(I-T^{\ast}T)=1$. They also proved in \cite{Wu2}  with B. Mirman \cite{Mirman1} separately that if $T$ is in $\Upsilon_{n}$ then $\partial W (T)$ satisfies 
the so-called $n+1$-Poncelet property. Recall that for $n\geq3$ we say that a curve $\varGamma$ satisfies the $n$-Poncelet property
if for every $\lambda$ on the unit circle there is an $n$-gone which circumscribes about $\varGamma$, inscribes in the unit circle 
and has $\lambda$ as a vertex.

\begin{theorem}[\cite{Wu1} Theorem 5.1 or \cite{Wu2} Theorem 2.1 ]\label{ter1}
\textit{ For any matrix $T$ in $\Upsilon_{n}$ and any point $\lambda$ with $|\lambda|=1$; there is a
unique $(n + 1)$-gon which circumscribes about $\partial W (T)$; inscribes in the unit circle and has $\lambda$
as a vertex. In fact; such $(n + 1)$-gons $P$ are in one-to-one correspondence
with the (unitary-equivalence classes of) unitary dilations $U$ of $A$ on an $(n + 1)$-
dimensional space; under which the $n+ 1$ vertices of P are exactly the eigenvalues
of the corresponding U.}
\end{theorem}
Theorem \ref{ter1} yields additional properties for the numerical ranges of matrices in $\Upsilon_{n}$.
\begin{corollary}[Corollary 5.2 \cite{Wu2}]\label{ter2}
\textit{Let $T$ be a matrix in $\Upsilon_{n}$. Then}
\begin{enumerate}
  \item \textit{$W(T)$ is contained in no m-gon inscribed in the unit circle for $m\leq n$.}
\item \textit{$w_{2}(T)>\cos(\pi/n)$.}
\item \textit{$\mathscr{R}e(T)$ and $\mathscr{I}m(T)$ have simple eigenvalues.\label{jj}}
\item \textit{The boundary of $W(T)$ contains no line segment and is an algebraic curve.}
 \end{enumerate}
\end{corollary}
The reader may consult \cite{Halmos} chapter 22 for properties of numerical ranges of operators in general, \cite{Horn} chapter 1
 for those of finite dimensional operators and particularly \cite{Wu2} for the geometric properties of the numerical range 
of $\bbS(\phi)$.

\hspace{0.5cm}The numerical radius and the numerical range of of the model operator $\bbS(\phi)$ seems to be important and have many applications.
In \cite{Gaaya1}, the author showed that there is relationship between numerical radius of $\bbS(\phi)$ and Taylor coefficients of positive
rational functions on the torus which extends a previous result of C. Badea and G. Cassier (\cite{Cassier} Theorem 5.1). This result is formulated as the following:
\begin{theorem}[\cite{Gaaya1} Theorem 2.1]
 \textit{Let $F=P/Q$ be a rational function which is positive on the torus, where ${P}$ and ${Q}$ are coprime. 
Denote by $$\phi(z)=\prod_{j=1}^{p}\left( \dfrac{z-\alpha_{j}}{1-\overline{\alpha_{j}}z}\right)^{m_{j}} $$ and 
$$\psi(z)=\prod_{j=1}^{q}\left( \dfrac{z-\beta_{j}}{1-\overline{\beta_{j}}z}\right)^{d_{j}} $$ the respectively finite 
Blaschke products formed by the nonzero roots of $P$ and $Q$ in the open disc, let $m=\sum_{j=1}^{p}m_{j}$ and 
$d=\sum_{j=1}^{q}d_{j}$. Then the Taylor coefficient $c_{k}$ of order $k$ of $F$ satisfies the following inequality:
 \begin{eqnarray*}
  \lvert c_{k} \lvert \leqslant c_{0}~\omega_{2} ({\bbs}^{k}(\varphi)),~ \mbox{where}~ \varphi(z)=z^{\max(0,m-d+1)}\psi(z).
 \end{eqnarray*}
}
\end{theorem}
However, the evaluation of the numerical radius of $\bbS(\phi)$ under an explicit form is always an open problem. Which
explain the motivation of our first main result. In the section \ref{aa}, we
give an explicit formula of the numerical radius of $\bbS(\phi)$ in the particular case where $\phi$ is a finite Blaschke product with 
unique zero $\alpha$ in the unit disc:$$\phi(z)=\phi_{\alpha}(z)=\left( \dfrac{z-\alpha}{1-\overline{\alpha}z}\right)^n.$$
Our result, officially stated as Theorem \ref{Haykoulita}, is:
$$\omega_{2}(S(\phi))=\frac{-(1+|\alpha|^{2})\cos t_{n}^{(n)}+2|\alpha|}{1-2|\alpha| \cos t_{n}^{(n)}+|\alpha|^{2}}=
\dfrac{1-|\alpha|^{2}}{2\alpha}\Big(-P_{|\alpha|}(e^{it_{n}^{(n)}})+\dfrac{1+|\alpha|^{2}}{1-|\alpha|^{2}}\Big).$$
Here $P_{|\alpha|}(e^{it})=\sum_{k\in\mathbb{Z}}|\alpha|^{|k|}e^{ikt}$ is the Poisson kernel and $t_{n}^{(n)}$ is a precise point in 
the interval $]\frac{(n-1)\pi}{n+1},\frac{(n)\pi}{n+1}]$. In the section 3.1, we shall see how the celebrated Toeplitz matrix of 
Kac, Murdokh and Szegö (the Toeplitz matrix with the Poisson kernel as symbol) will play an important role to obtain this result.
We defer the proof of Theorem \ref{Haykoulita} to the section 3.2., preferring to devote the remainder of Section 3 to its 
corollaries. In the general case where $\phi$ is an arbitrary finite Blaschke
 product, an estimate of the numerical radius is given in section \ref{ee}.

\hspace{0.5cm}A celebrated inequality due to von Neumann \cite{Neumann} asserts that for $$\rVert p(T)\rVert\leq \rVert p\rVert_{\infty},$$
 for all polynomial $p\in\mathbb{C}[X]$. The same inequality holds for functions in the disc algebra $\mathbb{A(\mathbb{D})}$. More general,
if $T$ is a completely non-unitary (c.n.u) contraction, this result extends to bounded analytic functions $f\in\mathbb{H}^{\infty}$ \cite{Nagy}. Ptak and Young 
has also proved:
\begin{theorem}[\cite{Ptak}]
\textit{Suppose that $p$ and $q$ are arbitrary analytic polynomials and $T$ be a Hilbert space contraction in $\mathscr{B}(\mathcal{H})$ 
 such that the spectral radius $r(T)<1$ and $q(T)=0$. Then $$\rVert p(T)\rVert \leq \rVert p(S^{\ast}|\mbox{Ker}~~ q(S^{\ast}))\rVert.$$}
\end{theorem}
The condition $r(T)<1$ is not indispensable and the following generalisation was given by Sz.-Nagy
\begin{theorem}[\cite{Nagy3}]
 \textit{ Let $f$ and $g$ be two functions in $\mathbb{H}^{\infty}$ and $T$ be a c.n.u contraction in $\mathscr{B}(\mathcal{H})$ such that
 $g(T)=0$. Then 
$$\rVert f(T)\rVert \leq \rVert f(S^{\ast}|\mbox{Ker}~~ g(S^{\ast}))\rVert.$$}
\end{theorem}
In 1992, U. Haagerup and P. de la Harpe proved, using solely elementary methods (positive definite Kernels) that:
\begin{theorem}[\cite{Haagerup}]\label{Haa}
\textit{ Let $T$ be a Hilbert space contraction in $\mathscr{B}(\mathcal{H})$ such that $T^{n}=0$ for some $n \geq 2$. One has:
$$\omega_{2}(T)\leqslant \lVert T\lVert \omega_{2}( S_{n})=\lVert T\lVert \cos\frac{\pi}{n+1}.$$
Further $\omega_{2}(T)=\lVert T\lVert \cos\frac{\pi}{n+1}$ when $T$ is unitarily equivalent to 
$ S_{n}$}.
\end{theorem}
Apparently, there is not a relationship between the Haagerup and de la Harpe result and the von Neumann inequalities. But, their 
inequality states that if $u_{n}(z)=z^{n}$ and $u_{n}(T)=0$, then $\omega_{2}(T)\leq w_{2}(\bbs\lvert Ker~u_{n}({\bbS}^{\ast}))$.  In \cite{Cassier}, C. Badea and G. Cassier obtained constrained von Neumann inequalities which allow to see the previous result
as a corollary. Here is a simplified version of a constrained von Neumann inequality that we find in their article \cite {Cassier}. 
\begin{theorem}[\cite{Cassier}]\label{catalin}
\textit{ Let $T\in\mathscr{B}(\mathcal{H})$ be a contraction of class $C_{0}$ with $u(T)=0$, where $u$ be an inner function and let
 $f$ be in $\mathbb{A(D)}$.
Then $$w_{\rho}(f(T))\leq w_{\rho}(f(S(u))).$$ }
\end{theorem}
Here $w_{\rho}(T)$ denote the $\rho$-numerical radius of an operator $T$. Moreover, we shall use this version later in section \ref{dd} to give the second main result (Theorem \ref{habibi}) of this paper 
which consists on a sharpened Schwarz-Pick operatorial inequality for nilpotent operators.

\hspace{0.5cm}Before proceeding further, I would like to express my gratitude to Gilles Cassier for his help, suggestions and his good advices.

\section{The numerical radius of the truncated Shift $S(\phi)$ where $\phi$ is a Blaschke product with unique zero}
\subsection{Preliminary}
\hspace{0.5cm}For any $n$-by-$n$ matrix $T$ and real $\theta\in [0,2\pi[$ let $L$ be the supporting line of the convex set $W(T)$ perpendicular 
to the ray which emanates from the origin and forms angle $\theta$ from the positive $x$-axis and let 
 $\lambda=\lambda(\theta)$ be the signed distance from the origin to the line $L$. It is easily seen that $\lambda(\theta)$ is the 
 largest eigenvalue of  the Hermitian matrix 
$\mathscr{R}e(e^{-i\theta}T):$

$$D_{n}(\lambda,\theta)= \mbox{det}(\mathscr{R}e(e^{-i\theta}T-\lambda I_{n})) = 0.$$

Note that $\lambda(\theta)$ is taken with “$+$” if the origin is not separated from $\partial W(T)$ by the line, and taken 
with  “$-$”  otherwise.

The equation of the boundary $\partial W(T)$ of the numerical range of $T$ is defined by the maximum eigenvalue 
$\lambda=\lambda(\theta)$ of 
$$D_{n}(\lambda,\theta)= \mbox{det}(\mathscr{R}e(e^{-i\theta}T-\lambda I_{n})) = 0.$$

Namely, let the points $(x(\theta),y(\theta))$ of $\partial W(T)$  be parametrised by the angles
$0\leq\theta < 2\pi$ between the straight lines of support L and $x$-axis (See Figures 1 and 2).

$\partial W(T)$ is the envelope of chords
$x(\theta) \cos\theta + y(\theta) \sin \theta = \lambda(\theta)$.
Then \cite{Mirman2}, \cite{Kip}:

$$x=x(\theta)=\lambda(\theta)\cos\theta-\lambda'(\theta)\sin\theta$$
and
$$y=y(\theta)=\lambda(\theta)\sin\theta+\lambda'(\theta)\cos\theta$$

The derivative $\lambda'(\theta) $ is determined for the so-called"regular arcs" of $\partial W(T)$ . 
If $(x(\theta),y(\theta))$ is a point of a regular arc, then
$\lambda''(\theta)+\lambda'(\theta) $ is the radius of curvature of the arc at this point. A regular arc
of $\partial W(T)$ contains neither corner points nor straight line segments.
\begin{theorem}\label{hayka}
\textit{ If $T$ is in $\Upsilon_{n}$, then  $\partial W(T)$ is a regular arc and $\lambda(\theta) $ is differentiable for all
 $\theta\in [0.2\pi[$.}
\end{theorem}
We will prove Theorem \ref{hayka} through a series of lemmas and propositions, the first of which concerns a property
 of a regular arc.

\begin{proposition}[\cite{Mirman1} Proposition 2]\label{Mirmana1}
\textit{ If a connected subset of $\partial W(T)$ does not contain corner points and,
for all points of this subset, $\lambda(\theta) $ is a simple eigenvalue of $\mathscr{R}e(e^{-i\theta}T)$, then this
subset is a regular arc of $\partial W(T)$.}
\end{proposition}

\begin{lemma}[\cite{Horn} pp. 50-51]\label{Mirmana2}
\textit{If the boundary of the numerical range of a matrix $T$ contains a corner point $\lambda=<Tu,u>$, 
then $\lambda$ is a normal eigenvalue of $T$: $$Tu=\lambda u ~~\mbox{and}~~T^{\ast}u=\overline{\lambda}u.$$}
\end{lemma}
\begin{lemma}[\cite{Mirman1} Lemma 1]\label{Mirmana3}
\textit{If $U$ is an $(n+1)\times(n+1)$ unitary matrix with distinct eigenvalues, $Q=I_{n+1}-w\otimes w, \rVert w\rVert=1, n>1$ 
and $T=QUQ$ then the following assertions are equivalent.}
\begin{enumerate}
\item \textit{$w_{2}(T)<1$.} 
\item \textit{If $\lambda$ is an eigenvalue of $T$, then $|\lambda|<1$.}
\item \textit{$<w,u>\neq0$ for any eigenvector $u$ of $U$.}
\item \textit{The subspace $\mathcal{L}=Q\mathcal{H}_{n+1}$ contains no eigenvectors of $U$.}
\item \textit{$T$ does not have normal eigenvalues.}
\end{enumerate}
\end{lemma}
\begin{proof}[Proof of Theorem \ref{hayka}]
Let $T$ in $\mathscr{B}(\mathcal{H})$ and dim$\mathcal{H}=n$. If $T$ is in $\Upsilon_{n}$, so is $e^{-i\theta}T$ for any real $\theta$. Hence the eigenvalues of 
$\mathscr{R}e(e^{-i\theta}T)$ are all distinct by Corollary \ref{ter2}. Let $U$ an $(n+1)\times(n+1)$ unitary dilatation of $T$ 
on a the space $\mathcal{H}_{n+1}$ of dimension $n+1$
which contains $\mathcal{H}$, there is an orthonormal basis $\{e_{1},\dots,e_{n+1}\}$ of $\mathcal{H}_{n+1}$ such that 
$\{e_{1},\dots,e_{n}\}$ forms a basis of $\mathcal{H}$. let $U$ have this matrix representation with respect to the basis
$\{e_{1},\dots,e_{n+1}\}$:
$$
\left(
\begin{array}{cc}
   T & a \\
  
   b & c
\end{array}
\right),$$
where $a, b\in \mathbb{C}^{n}$ and $c\in \mathbb{C}.$
By Theorem \ref{ter1}, the $n+ 1$ eigenvalues of U are distinct. On the other hand, we can easily check that if
$Q=I_{n+1}-e_{n+1}\otimes e_{n+1}$, then $QUQ=T$. An application of Proposition \ref{Mirmana1}, Lemma \ref{Mirmana2}
 and Lemma \ref{Mirmana3} completes the proof of the theorem.
\end{proof}

\begin{proposition}\label{nika1}
\textit{Let $T\in\mathscr{B}(\mathcal{H})$ with a regular boundary arc. One has:
$$\omega_{2}(T)=\sup\{\lambda(\theta),~ 0\leq\theta < 2\pi\}.$$}
\end{proposition}
\begin{proof}
Recall that for any $T\in\mathscr{B}(\mathcal{H})$ we have
\begin{eqnarray*}
 \omega_{2}(T)&=&\sup\{\rVert \mathscr{R}e(e^{-i\theta}T)\rVert,~ 0\leq\theta < 2\pi\}.
\end{eqnarray*}
If $\partial W(T)$ is a regular arc then 
\begin{eqnarray*}
 \omega_{2}(T)
&=&\sup\{|\lambda(\theta)|,~ 0\leq\theta < 2\pi\}.
\end{eqnarray*}
Now, when the numerical radius is attained, the origin is not separated from $\partial W(T)$ by the line $L$ and necessarily 
$\lambda(\theta)$ is positive.
\end{proof}
\subsection{Relationship between numerical radius of $S(\phi)$ and $\mathscr{R}e(S(\phi))$}

First, we notice some properties for the general case where $\phi$ is a finite Blaschke product: $$\phi(z)= \prod_{j=1}^{n}~\dfrac{z-\alpha_{j}}{1-\overline{\alpha_{j}}z}.$$ For each $\lambda$ in the unit disc, we define the evaluation functional $k_{\lambda} \in \H2$ by the requirement that $f(\lambda)=<f,k_{\lambda}> $. Thus$$k_{\lambda}(z)=\dfrac{1}{1-\overline{\lambda}z}$$ and  $\left\lbrace e_{1}, \dots,e_{n} \right\rbrace $ the collection of functions of $H(\phi)$ defined as follows :
\begin{equation*}
e_{1}(z)= \left(1-\vert\alpha_{1}\vert^{2}\right)^{\frac{1}{2}}  ~\dfrac{1}{1-\overline{\alpha_{1}}z}
\end{equation*}
 and
$$e_{k}(z)= \left(1-\vert\alpha_{k}\vert^{2}\right)^{\frac{1}{2}}~\dfrac{1}{1-\overline{\alpha_{k}}z}~~ \prod_{j=1}^{k-1}\dfrac{z-\alpha_{j}}{1-\overline{\alpha_{j}}z}$$ for any $k=2,...,n$.

It is known that $\left\lbrace e_{1}, \dots,e_{n} \right\rbrace $ is an orthonormal basis of $H(\phi)$ and with respect to this basis the matrice of $\bbs(\phi)$ is given by $\left[ a_{lk} \right]$, where
 $$ a_{lk}=\left\{
    \begin{array}{ll}
         \overline{\alpha_{l}} & \mbox{if } l=k \\ 
         \sigma_{l}\sigma_{l+1}& \mbox{if } k=l+1 \\ 
         \sigma_{l}\sigma_{k} \prod_{j=l+1}^{k-1} (-\alpha_{j}) & \mbox{if } k>l+1\\
        0 & \mbox{unless}
    \end{array}
\right.$$
and $ \sigma_{k}=\left(1-\vert\alpha_{k}\vert^{2}\right)^{\frac{1}{2}}$, for each $1\leqslant k \leqslant n$.
In the sequel of the paper, $\phi$ denotes the finite Blaschke product with unique zero $\alpha$:
$$
\phi(z)=\phi_{\alpha}(z)=\left( \dfrac{z-\alpha}{1-\overline{\alpha}z}\right)^n.
$$
$S^{\ast}(\phi_{\alpha})$ gets the following matricial representation:
$$
\left(
\begin{array}{cccccc}
\overline{\alpha}&\sigma           &-\sigma\alpha &\cdots &\cdots           &\sigma(-\alpha)^{n-2} \\
0                &\overline{\alpha}& \sigma       &\ddots &                 &\vdots \\
\vdots           &       \ddots    & \ddots       &\ddots &\ddots           &\vdots\\
\vdots           &                 & \ddots       &\ddots &\ddots           &-\sigma\alpha\\
\vdots           &                 &              &\ddots &\overline{\alpha}&\sigma\\
0                &       \dots     &  \dots       &\dots  &         0       &\overline{\alpha}
\end{array}
\right)
$$
where $ \sigma=1-\vert \alpha\vert^{2}$.
\begin{proposition}\label{nika}
\textit{For $\alpha\in\mathbb{C}$ and $\lvert\alpha\lvert<1$, one has:}
\begin{enumerate}
\item  \textit{$S^{\ast}(\phi_{\alpha})=(S_{n}^{\ast}+\overline{\alpha}I_{n})(I_{n}+\alpha S_{n}^{\ast})^{-1}.$}
\\
\item  \textit{$W(S^{\ast}(\phi_{\alpha}))=e^{-i \arg(\alpha)}W(S^{\ast}(\phi_{|\alpha|})).$}
\\
\item  \textit{The numerical radius of $S^{\ast}(\phi_{\alpha})$ is independent from the argument of $\alpha$ and 
for $0 \leqslant \alpha<1 $ the numerical range of  $S^{\ast}(\phi_{\alpha})$ is symmetric with respect to the real axis.}
\end{enumerate}
\end{proposition}
\begin{proof}
 Here (1) is due to the fact that 
\begin{eqnarray*}
S^{\ast}(\phi_{\alpha})&=& \overline{\alpha} I_{n}+\sigma \sum_{k=0}^{n-2}(-\alpha)^{k}{S^{\ast}_{n}}^{k}\\
                 &=& \overline{\alpha} I_{n}+\sigma \sum_{k=0}^{\infty}(-\alpha)^{k}{S^{\ast}_{n}}^{k}\\
                 &=& \overline{\alpha} I_{n}+\sigma(I_{n}+\alpha S^{\ast}_{n})^{-1}S^{\ast}_{n}\\
                 &=& (S_{n}^{\ast}+\overline{\alpha}I_{n})(I_{n}+\alpha S_{n}^{\ast})^{-1}.
\end{eqnarray*}
The assumption (2) is a consequence of the fact that $S^{\ast}(\phi_{\alpha})$ is the n-Toeplitz matrix associated to the Toeplitz form 
$\dfrac{e^{-it}
+\overline{\alpha}}{1+\alpha e^{-it}}$. For 
$u=(u_{0},\dots,u_{n-1})$ in ${\C}^{n}$ with ${\rVert u\rVert}_{2}=1$, we have 
\begin{eqnarray*}
 <S^{\ast}(\phi_{\alpha})u,u> &=&\int_{-\pi}^{\pi}\dfrac{e^{-it}+\overline{\alpha}}{1+\alpha e^{-it}}\Big|\sum_{k=0}^{n-1}u_{k}e^{ikt}\Big|^{2}\dfrac
{dt}{2\pi}\\
 &=& e^{-i\arg(\alpha)}\int_{-\pi}^{\pi}\dfrac{e^{-it}+|\alpha|}{1+|\alpha| e^{-it}}\Big|\sum_{k=0}^{n-1}v_{k}e^{ikt}\Big|^{2}\dfrac
{dt}{2\pi}\\
 &=& e^{-i\arg(\alpha)} <S^{\ast}(\phi_{|\alpha|})v,v>,
\end{eqnarray*}
with $v_{k}=e^{ik\arg(\alpha)}u_{k}$ and $v=(v_{0},\dots,v_{n-1}).$

(3) is in \cite{Gaaya1}.
\end{proof}

So, from Proposition \ref{nika}, the study of the numerical radius of $S^{\ast}(\phi_{\alpha})$ is independent from
 the argument of $\alpha$, and that, more generally, the numerical radius of $S^{\ast}(\phi_{\alpha})$ should be connected with 
the numerical radius of its real part. In the sequel, we will show that this intuition is correct.

For all $n\geq2$, we denote by $D_{n}(\lambda,\theta)$ the characteristic polynomial of $S^{\ast}(\phi_{-\alpha})$ where $0\leq \alpha<1$.
\begin{lemma}
\textit{ For all $n\geq2$, one has 
\begin{eqnarray*}
 D_{n}(\lambda,\theta)&=&\dfrac{(1-\lambda^{2})^{-\frac{1}{2}}}{2^n}~\mathscr{R}e\Bigg(\left(\left(1-\lambda^{2}\right)^{\frac{1}{2}}+i\lambda\right)\\
& &\left(-2\alpha\cos\theta-(1+\alpha^{2})\lambda+i(1-\alpha^{2})(1-\lambda^{2})^{\frac{1}{2}}\right)^n\Bigg).
\end{eqnarray*}}

\end{lemma}
\begin{proof}
 For all $n\geq2$, we have
\[ D_{n}(\lambda,\theta) = \left|
\begin{array}{ccccc}
-\alpha\cos\theta-\lambda& \dfrac{\sigma}{2}e^{-i\theta} &  \dfrac{\alpha\sigma}{2}e^{-i\theta} &\cdots& \dfrac{\alpha^{n-2}\sigma}{2}e^{-i\theta} \\

\dfrac{\sigma}{2}e^{i\theta} &-\alpha\cos\theta-\lambda & \dfrac{\sigma}{2}e^{-i\theta}&\cdots& \dfrac{\alpha^{n-3}\sigma}{2}e^{-i\theta}\\

\dfrac{\alpha\sigma}{2}e^{i\theta} & \dfrac{\sigma}{2}e^{i\theta} & -\alpha\cos\theta-\lambda&\cdots&\dfrac{\alpha^{n-4}\sigma}{2}e^{-i\theta} \\

\cdots&\cdots&\cdots&\cdots&\cdots\\

\dfrac{\alpha^{n-2}\sigma}{2}e^{i\theta}&\dfrac{\alpha^{n-3}\sigma}{2}e^{i\theta}&\dfrac{\alpha^{n-4}\sigma}{2}e^{i\theta}&\cdots&-\alpha\cos\theta-\lambda
\end{array} \right|. \]
Multiplying the second row of this determinant by $\alpha$ and subtracting it from the first we obtain
$$D_{n}(\lambda,\theta)=\left|
\begin{array}{ccccc}
    a(\lambda,\theta)& c(\lambda,\theta) &  0 &\cdots& 0\\
   \dfrac{\sigma}{2}e^{i\theta} &-\alpha\cos\theta-\lambda & \dfrac{\sigma}{2}e^{-i\theta}&\cdots& \dfrac{\alpha^{n-3}\sigma}{2}e^{-i\theta}\\
   \dfrac{\alpha\sigma}{2}e^{i\theta} & \dfrac{\sigma}{2}e^{i\theta} & -\alpha\cos\theta-\lambda&\cdots&\dfrac{\alpha^{n-4}\sigma}{2}e^{-i\theta} \\
\cdots&\cdots&\cdots&\cdots&\cdots\\
\dfrac{\alpha^{n-2}\sigma}{2}e^{i\theta}&\dfrac{\alpha^{n-3}\sigma}{2}e^{i\theta}&\dfrac{\alpha^{n-4}\sigma}{2}e^{i\theta}&\cdots&-\alpha\cos\theta-\lambda
\end{array} \right|,$$ 
with  $a(\lambda,\theta)=-\alpha\cos\theta-\lambda-\dfrac{\alpha\sigma}{2}e^{i\theta}$ and 
$c(\lambda,\theta)=\dfrac{\sigma}{2}e^{-i\theta}+\alpha^{2}\cos\theta+\alpha\lambda$.
Performing a similar operation with the columns, we find
$$D_{n}(\lambda,\theta)=\left|
\begin{array}{ccccc}
    b(\lambda,\theta)& c(\lambda,\theta) &  0 &\cdots& 0\\
   \overline{c(\lambda,\theta)}&-\alpha\cos\theta-\lambda & \dfrac{\sigma}{2}e^{-i\theta}&\cdots& \dfrac{\alpha^{n-3}\sigma}{2}e^{-i\theta}\\
   0 & \dfrac{\sigma}{2}e^{i\theta} & -\alpha\cos\theta-\lambda&\cdots&\dfrac{\alpha^{n-4}\sigma}{2}e^{-i\theta} \\
\cdots&\cdots&\cdots&\cdots&\cdots\\
0&\dfrac{\alpha^{n-3}\sigma}{2}e^{i\theta}&\dfrac{\alpha^{n-4}\sigma}{2}e^{i\theta}&\cdots&-\alpha\cos\theta-\lambda
\end{array} \right|, $$
with $b(\lambda,\theta)=-2\alpha\cos\theta-\lambda(1+\alpha^{2})$.
Which implies that for $n\geq3$, we have 
$$D_{n}(\lambda,\theta)=\bigg(-2\alpha\cos\theta-\lambda(1+\alpha^{2})\bigg)D_{n-1}(\lambda,\theta)-\bigg|\dfrac{\sigma}{2}e^{i\theta}+\alpha^{2}\cos\theta+\alpha\lambda\bigg|^{2}D_{n-2}(\lambda,\theta).$$
This recurrence relation holds also for $n=2$ provided we put $D_{0}(\lambda,\theta)=1$ and form the equation 
$$\rho^{2}=-\bigg(2\alpha\cos\theta+\lambda(1+\alpha^{2})\bigg)\rho-\bigg|\dfrac{\sigma}{2}e^{i\theta}+\alpha^{2}\cos\theta+\alpha\lambda\bigg|^{2}$$
with discriminant $$\Delta=\bigg(\mp i(1-\alpha^{2})(1-\lambda^{2})^{\frac{1}{2}}\bigg)^{2}.$$ The roots are
$$\rho_{1}=\dfrac{-2\alpha\cos\theta-\lambda(1+\alpha^{2})- i(1-\alpha^{2})(1-\lambda^{2})^{\frac{1}{2}}}{2}$$and 
$$\rho_{2}=\dfrac{-2\alpha\cos\theta-\lambda(1+\alpha^{2})+ i(1-\alpha^{2})(1-\lambda^{2})^{\frac{1}{2}}}{2}.$$
So that
$$D_{n}(\lambda,\theta)=A{\rho_{1}}^{n}+B{\rho_{2}}^{n},$$
where the constants $A$, $B$ can be determined from the ``initial conditions``
$$D_{0}(\lambda,\theta)=1=A+B$$ and 
$$D_{1}(\lambda,\theta)=-\alpha\cos\theta-\lambda=A{\rho_{1}}+B{\rho_{2}}.$$
This yields that $$A=\dfrac{(1-\lambda^{2})^\frac{1}{2}-i\lambda}{2(1-\lambda^{2})^\frac{1}{2}}, ~~~B=\dfrac{(1-\lambda^{2})^\frac{1}{2}+i\lambda}{2(1-\lambda^{2})^\frac{1}{2}}$$
which establishes the desired formula. 
\end{proof}
Using this Lemma, we obtain the following result.
\begin{proposition}\label{bb1}
\textit{ One has:
$$\omega_{2}(S^{\ast}(\phi_{-\alpha}))=\omega_{2}(\mathscr{R}e(S^{\ast}(\phi_{-\alpha}))).$$}
\end{proposition}
\begin{proof}
 Let consider the applications:
$$ \begin{array}{ll}
\Phi_{n} : &[0,\pi] \longrightarrow \mathbb{R} \\
   &\theta \longrightarrow D_{n}(\lambda(\theta),\theta)\end{array}$$
 $$ \begin{array}{ll}
\Psi_{n} : &\mathbb{R}^{2} \longrightarrow \mathbb{R} \\
   &(x,y) \longrightarrow D_{n}(x,y)\end{array}$$
 and
$$ \begin{array}{ll}
\Omega : &[0,\pi]  \longrightarrow \mathbb{R}^{2} \\
   &\theta \longrightarrow (\lambda(\theta),\theta)
\end{array}.$$
Now, since $\lambda(\theta)$ is an eigenvalue of $\mathscr{R}e(e^{-i\theta}S^{\ast}(\phi_{-\alpha}))$ thus 
 $\Phi_{n}(\theta)=\Psi_{n}\circ\Omega(\theta)=0$ 
for all $\theta\in[0,\pi]$ which implies that 
$$\Phi_{n}'(\theta)=D\Psi_{n}(\Omega_{n}(\theta))(\lambda'(\theta),1)=d'(\theta)\dfrac{\partial\Psi_{n}}{\partial x}(\lambda(\theta),\theta)
+\dfrac{\partial\Psi_{n}}{\partial y}(\lambda(\theta),\theta)=0.$$ Recall that from Proposition \ref{nika},
 $W(S^{\ast}(\phi_{-\alpha}))$ is symmetric with respect to the real axis. Therefore, using Proposition \ref{nika1},
$$\omega_{2}(S^{\ast}(\phi_{-\alpha}))=\sup\{\lambda(\theta),~ 0\leq\theta \leq\pi\}.$$
Assume that there exists $\theta_{0}\in]0,\pi[$ such that $d'(\theta_{0})=0$ then $\dfrac{\partial\Psi_{n}}{\partial y}(d(\theta_{0}),\theta_{0})=0$. Therefore
$$2\alpha n\sin\theta_{0}\Phi_{n-1}(\theta_{0})=0$$
Now since $\sin\theta_{0}\neq0$ then $D_{n-1}(d(\theta_{0}),\theta_{0})=0$ which is impossible. Otherwise $1=D_{0}(d(\theta_{0}),\theta_{0})=0$.
Hence, $\lambda(\theta)$ 
is a monotonic function and its maximum is attained for $\theta=0$ or $\pi$. This 
concludes the proof.
\end{proof}
\section{An explicit formula of the numerical radius of $S(\phi)$ where $\phi$ is a finite Blaschke product with unique zero.}
\subsection{Preliminaries}
Toeplitz matrices are found in several areas of mathematics such complex and harmonic analysis. One of these matrices 
is of particular interest in these areas. It is about the Kac, Murdokh and Szegö matrix \cite{Kac}:
$$K_{n}(\alpha)=\left(
\begin{array}{cccc}
1&\alpha&\cdots&\alpha^{n-1}\\
\alpha&\ddots&\ddots &  \vdots                \\
\vdots&  \ddots               &  \ddots                      &\alpha          \\
\alpha^{n-1}&  \cdots               &   \alpha                     &        1
\end{array}
\right)=(\alpha^{\lvert r-s\lvert})^{n}_{r,s=1}$$ 
with $0\leqslant\alpha<1$.

The spectral decomposition of this matrix is very well understood in the computational sense. For this reason, 
these matrices are often used as test matrices. 
It's shown in \cite{Szego} page 69--72 that $K_{n}(\alpha)$ is a Toeplitz matrix associated with the Poisson kernel $P_{\alpha}(e^{it})=(1-\alpha^{2}/{\vert 1-\alpha e^{it}\vert}^{2}$ and its eigenvalues are:$$ \lambda_{k}^{(n)}=P_{\alpha}(e^{it_{k}^{(n)}})~~,1\leq k \leq n$$
where $t_{k}^{(n)}$ are the solutions of 
\begin{equation}
 p_{n}(\cos t)=\dfrac{\sin(n+1)t-2\alpha\sin nt+\alpha^{2}\sin(n-1)t}{\sin t}=0.
\end{equation}
The expression $p_{n}(\cos t)$ is a polynomial of degree $n$ in $\cos t$ and it has $n$ real distinct zeros $\cos t_{k}^{(n)}$ for $1\leqslant k\leqslant n$ where :
$$0< t_{1}^{(n)} <  t_{2}^{(n)} < t_{3}^{(n)} < \cdots  < t_{n}^{(n)}< \pi~.$$
This implies that $$\dfrac{1+\alpha}{1-\alpha} >\lambda_{1}^{(n)} > \lambda_{2}^{(n)} > \lambda_{3}^{(n)} > \cdots > \lambda_{n}^{(n)} > \dfrac{1-\alpha}{1+\alpha} ~. $$
The evaluation of the zeros $t_{k}^{(n)}$ in explicit terms seems to be out of end. However, it is easy to show that they are separated by the quantities $x_{k}=\dfrac{k\pi}{n+1},~1\leq k \leq n$. Indeed, for $1\leqslant k\leqslant n$  $$p_{n}(\cos x_{k})=(-1)^{k}2\alpha(1-\alpha\cos x_{k})$$ and  $$ sgn~p_{n}(\cos x_{k})=(-1)^{k}~.$$
Also we see by direct substitution that the latter equation holds for $k=0$, so that $$0< t_{1}^{(n)} \leqslant x_{1}<  t_{2}^{(n)}\leqslant x_{2}  < \cdots  < t_{n}^{(n)}\leqslant x_{n}< \pi ~.$$

\begin{remark}\label{bsissa}
\textit{In the case where $\alpha=0$ we have $t_{k}^{(n)}=x_{k}$. }
\end{remark}

\begin{proposition}\label{Haykoul117}
 
\textit{For all $1\leq k \leq n$, we have $$t_{k}^{(n)} ~~\mbox{is the unique solution of }~~\left\{
    \begin{array}{ll}
        \cos\tfrac{(n+1)t}{2}-\alpha \cos\tfrac{(n-1)t}{2}  & \mbox{if}~~ k~~ \mbox{is odd} \\
        \sin\tfrac{(n+1)t}{2}-\alpha \sin\tfrac{(n-1)t}{2}& \mbox{if}~~ k ~~\mbox{is even.}
    \end{array}
\right.
$$ in the interval $]\tfrac{(k-1)\pi}{k+1}, \tfrac{k\pi}{k+1}]$.}
\end{proposition}
\begin{proof}
First of all, we observe that 
$$p_{n}(\cos t)=\frac{2}{\sin t}\left( \sin\frac{(n+1)t}{2}-\alpha \sin\frac{(n-1)t}{2}\right) \left( \cos\frac{(n+1)t}{2}-\alpha \cos\frac{(n-1)t}{2}\right).$$
Note that for all $1\leq k \leq n$, we have
$$\dfrac{(k-1)\pi}{k+1}\leq t_{k}^{(n)}\leq \dfrac{k\pi}{k+1}$$
Then
$$(k-1)\dfrac{\pi}{2}\leq \dfrac{n+1}{2}t_{k}^{(n)}\leq \dfrac{k}{2}\pi.$$We consider two cases:

\underline{If $k$ is even:}

Then $k=2p$ with $p\in \mathbb{N}^{\ast}$ and 
$$(p-\dfrac{1}{2})\pi\leq \dfrac{n+1}{2}t_{k}^{(n)}\leq p\pi$$ which implies that
 $$\cos\frac{(n+1)t_{k}^{(n)}}{2}~~\sin\frac{(n+1)t_{k}^{(n)}}{2}\leq0.$$ 
If $t_{k}^{(n)}$ is a solution of the equation $\cos\frac{(n+1)t}{2}=\alpha \cos\frac{(n-1)t}{2}$. Therefore
$$\alpha\left(\cos\frac{(n+1)t_{k}^{(n)}}{2}\cos t_{k}^{(n)}+ \sin\frac{(n+1)t_{k}^{(n)}}{2}\sin t_{k}^{(n)}\right)=
\cos\frac{(n+1)t_{k}^{(n)}}{2}$$
Thus
\begin{equation}
 \cos\frac{(n+1)t_{k}^{(n)}}{2}\left(1-\alpha\cos t_{k}^{(n)}\right)=\alpha \sin\frac{(n+1)t_{k}^{(n)}}{2}\sin t_{k}^{(n)}\label{1}
\end{equation}
Since $1-\alpha\cos t_{k}^{(n)}$ and $\sin t_{k}^{(n)}$ are both positive then from (\ref{1}), $\cos\dfrac{(n+1)t_{k}^{(n)}}{2}$ and 
$\sin\dfrac{(n+1)t_{k}^{(n)}}{2}$ have inevitably the same sign which is absurd.

\underline{If $k$ is odd:}

We have
$$p\pi\leq \dfrac{n+1}{2}t_{k}^{(n)}\leq (p+\dfrac{1}{2})\pi$$ Thus
 $$\cos\frac{(n+1)t_{k}^{(n)}}{2}~~\sin\frac{(n+1)t_{k}^{(n)}}{2}\geq0.$$ 
If $t_{k}^{(n)}$ is a solution of the equation $\sin\frac{(n+1)t}{2}=\alpha \sin\frac{(n-1)t}{2}$. Then
$$\alpha\left(\sin\frac{(n+1)t_{k}^{(n)}}{2}\cos t_{k}^{(n)}- \cos\frac{(n+1)t_{k}^{(n)}}{2}\sin t_{k}^{(n)}\right)=
\sin\frac{(n+1)t_{k}^{(n)}}{2}$$
Thus
\begin{equation}
 \sin\frac{(n+1)t_{k}^{(n)}}{2}\left(\alpha\cos t_{k}^{(n)}-1\right)=\alpha \cos\frac{(n+1)t_{k}^{(n)}}{2}\sin t_{k}^{(n)}
\end{equation}
Since $\alpha\cos t_{k}^{(n)}-1\leq0$ and $\sin t_{k}^{(n)}\geq0$ then
 $$\cos\frac{(n+1)t_{k}^{(n)}}{2}~~\sin\frac{(n+1)t_{k}^{(n)}}{2}\leq0$$  which is impossible.

\end{proof}

\begin{proposition}\label{akouda}
 \textit{For $0\leqslant\alpha<1$, we have 
$$\omega_{2}(\mathscr{R}e(S^{\ast}(\phi_{-\alpha})))=\frac{-(1+\alpha^{2})\cos t_{n}^{(n)}+2\alpha}{1-2\alpha \cos t_{n}^{(n)}+\alpha^{2}}.$$}
\end{proposition}
\begin{proof}
 First, notice that where $ \alpha=0$, then
$$\mathscr{R}e(S^{\ast}(\phi_{-\alpha}))=\dfrac{1}{2}
\left(
\begin{array}{ccccc}
0     & 1   &  0   & 0    & \dots\\
1     & 0   &  1   & 0    &   \dots               \\
0     & 1   &  0   & 1    &\dots          \\
0     &  0  &   1  & 0    &\dots         \\
\dots &\dots&\dots & \dots &\dots

\end{array}
\right)
.$$
In this case the eigenvalues are $\cos\dfrac{k\pi}{n+1}$, for $k=1,\dots,n$. For the proof there are many references, 
we refer the reader for example to \cite{Szego} page 67 or \cite{Bottcher} page 35, 
therefore $$\omega_{2}(\mathscr{R}e(S^{\ast}(\phi_{-\alpha})))=\cos\dfrac{\pi}{n+1}.$$ Hence using Remark \ref{bsissa}, we 
observe that the Proposition \ref{akouda} is 
satisfied in the particular case where $ \alpha=0$.
Then we can limit our study to the case $\alpha\neq 0 $. Notice that 
$$\mathscr{R}e(S^{\ast}(\phi_{-\alpha}))=\dfrac{1-\alpha^{2}}{2\alpha}
\left(
\begin{array}{cccc}
-\dfrac{2\alpha^{2}}{1-\alpha^{2}}&\alpha&\cdots&\alpha^{n-1}\\
\alpha&\ddots&\ddots &  \vdots                \\
\vdots&  \ddots               &  \ddots                      &\alpha          \\
\alpha^{n-1}&  \cdots               &   \alpha                     &   -\dfrac{2\alpha^{2}}{1-\alpha^{2}}     
\end{array}
\right).
$$
Here $\mathscr{R}e(S^{\ast}(\phi_{-\alpha}))$ is the Toeplitz matrix associated with the Toeplitz form: 
$$\dfrac{1-\alpha^{2}}{2\alpha}\Big(P_{\alpha}(e^{it})-\dfrac{1+\alpha^{2}}{1-\alpha^{2}}\Big)=
\dfrac{(1+\alpha^{2})\cos t-2\alpha}{1-2\alpha \cos t+\alpha^{2}}=h(t).$$ 

Let $v=(v_{0},v_{1},\cdots v_{n-1})$ an unit vector in $\mathbb{C}^{n}$ where the 
numerical radius of $\mathscr{R}e(S^{\ast}(\phi_{-\alpha}))$ is attained. That is 
$<\mathscr{R}e(S^{\ast}(\phi_{-\alpha}))v,v>\lambda$ and 
$\omega_{2}(\mathscr{R}e(S^{\ast}(\phi_{-\alpha})))=|\lambda|$. Thus
\begin{eqnarray*}
 <\mathscr{R}e(S^{\ast}(\phi_{-\alpha}))v,v>&=&\int_{-\pi}^{\pi}\dfrac{1-\alpha^{2}}{2\alpha}(P_{\alpha}(e^{it})-\dfrac{1+\alpha^{2}}{1-\alpha^{2}})\Big|\sum_{l=0}^{n-1}v_{l}e^{ilt}\Big|^{2}\frac{dt}{2\pi}\\
&=&\int_{-\pi}^{\pi}\dfrac{1-\alpha^{2}}{2\alpha}(P_{\alpha}(e^{it})-\dfrac{1+\alpha^{2}}{1-\alpha^{2}})\Big|\sum_{l=0}^{n-1}\overline{v_{l}}e^{ilt}\Big|^{2}\frac{dt}{2\pi}\\
&=&<\mathscr{R}e(S^{\ast}(\phi_{-\alpha}))\overline{v},\overline{v}>
\end{eqnarray*}
where $\overline{v}=(\overline{v_{0}},\overline{v_{1}},\cdots \overline{v_{n-1}})$. According to Corollary \ref{ter2} (\ref{jj}), $\lambda$ is a simple eigenvalue of
$\mathscr{R}e(S^{\ast}(\phi_{-\alpha}))$ then there exists a real $\gamma$ such that $v=e^{i\gamma}\overline{v}$. Hence we may assume, by replacing $v$
by $e^{-i\frac{\gamma}{2}}v$, that $v=\overline{v}$. Then the numerical radius of  $\mathscr{R}e(S^{\ast}(\phi_{-\alpha}))$ 
is attained for a unit vector $v$ with real coefficients. Consequently, we obtain
\begin{eqnarray*}
<\mathscr{R}e(S^{\ast}(\phi_{-\alpha}))v,v>&=&\dfrac{1-\alpha^{2}}{2\alpha}(\sum_{l,m=0}^{n-1}\alpha^{|l-m|}v_{l}v_{m}-\dfrac{1+\alpha^{2}}{1-\alpha^{2}})\\
&=&\dfrac{1-\alpha^{2}}{2\alpha}\Big\{\sum_{0\leq l,m\leq n-1, l-m~\mbox{even}}\alpha^{|l-m|}v_{l}v_{m}-\dfrac{1+\alpha^{2}}{1-\alpha^{2}}\\
&+&\sum_{0\leq l,m\leq n-1, l-m~\mbox{odd}}\alpha^{|l-m|}v_{l}v_{m}\Big\}.
\end{eqnarray*}
Besides, we observe that
\begin{eqnarray*}
 0\leq \sum_{0\leq l,m\leq n-1, l-m~\mbox{even}}\alpha^{|l-m|}v_{l}v_{m}=
\int_{-\pi}^{\pi}P_{\alpha^{2}}(e^{2it})\Big|\sum_{l=0}^{n-1}v_{l}e^{ilt}\Big|^{2}\dfrac{dt}{2\pi}\leq\dfrac{1+\alpha^{2}}{1-\alpha^{2}}.
\end{eqnarray*}
This clearly forces 
\begin{eqnarray*}
|<\mathscr{R}e(S^{\ast}(\phi_{-\alpha}))v,v>|&\leq&\dfrac{1-\alpha^{2}}{2\alpha}\Big\{\dfrac{1+\alpha^{2}}{1-\alpha^{2}}-\sum_{0\leq l,m\leq n-1, l-m~\mbox{even}}\alpha^{|l-m|}v_{l}v_{m}\\
&+&\sum_{0\leq l,m\leq n-1, l-m~\mbox{odd}}\alpha^{|l-m|}|v_{l}||v_{m}|\Big\}.
\end{eqnarray*}
Consider $\widetilde{v}=(\widetilde{v_{0}},\widetilde{v_{1}},\widetilde{v_{2}},\cdots,\widetilde{v_{n-1}})=(|v_{0}|,-|v_{1}|,|v_{2}|,\cdots,(-1)^{n-1}|v_{n-1}|)$. Then 
\begin{eqnarray*}
<\mathscr{R}e(S^{\ast}(\phi_{-\alpha}))\widetilde{v},\widetilde{v}>&=&\dfrac{1-\alpha^{2}}{2\alpha}\Big\{\sum_{0\leq l,m\leq n-1, l-m~\mbox{even}}\alpha^{|l-m|}|v_{l}||v_{m}|-\dfrac{1+\alpha^{2}}{1-\alpha^{2}}\\
&-&\sum_{0\leq l,m\leq n-1, l-m~\mbox{odd}}\alpha^{|l-m|}|v_{l}||v_{m}|\Big\}.
\end{eqnarray*}
and hence
\begin{eqnarray*}
|<\mathscr{R}e(S^{\ast}(\phi_{-\alpha}))\widetilde{v},\widetilde{v}>|&=&\dfrac{1-\alpha^{2}}{2\alpha}\Big\{\dfrac{1+\alpha^{2}}{1-\alpha^{2}}-\sum_{0\leq l,m\leq n-1, l-m~\mbox{even}}\alpha^{|l-m|}|v_{l}||v_{m}|\\
&+&\sum_{0\leq l,m\leq n-1, l-m~\mbox{odd}}\alpha^{|l-m|}|v_{l}||v_{m}|\Big\}\\
&=&\dfrac{1-\alpha^{2}}{2\alpha}\Big\{\dfrac{1+\alpha^{2}}{1-\alpha^{2}}-\sum_{0\leq l,m\leq n-1, l-m~\mbox{even}}\alpha^{|l-m|}\widetilde{v_{l}}\widetilde{v_{m}}\\
&+&\sum_{0\leq l,m\leq n-1, l-m~\mbox{odd}}\alpha^{|l-m|}|\widetilde{v_{l}}||\widetilde{v_{m}}|\Big\}.
\end{eqnarray*}
It follows that the numerical radius of $\mathscr{R}e(S^{\ast}(\phi_{-\alpha}))$ is attained at $\widetilde{v}$ in the negative $x$-axis. Then
$$\omega_{2}(\mathscr{R}e(S^{\ast}(\phi_{-\alpha})))=-\lambda$$  
where $\lambda$ is the infimum of the eigenvalues of $\mathscr{R}e(S^{\ast}(\phi_{-\alpha}))$.
To complete the proof of the proposition, a straightforward argument, based on the fact that if $a$ and $b$ are arbitrary real number
 and $f(x)$ a Toeplitz form with $\gamma_{k}^{n}$ as eigenvalues then the eigenvalues of $a+bf(x)$ will be $a+b\gamma_{k}^{n}$. 
This shows that the eigenvalues of $\mathscr{R}e(S^{\ast}(\phi_{-\alpha}))$ are the $(\lambda_{k}^{(n)})_{1\leq k\leq n}$ with
$$\lambda_{k}^{(n)}=\dfrac{1-\alpha^{2}}{2\alpha}(P_{\alpha}(e^{it_{k}^{(n)}})-\dfrac{1+\alpha^{2}}{1-\alpha^{2}})=
\dfrac{(1+\alpha^{2})\cos t_{k}^{(n)}-2\alpha}{1-2\alpha \cos t_{k}^{(n)}+\alpha^{2}}.$$ Now, since $h(t)$ is monotonic 
on $[0,\pi]$, we may assume that: 
\begin{eqnarray*}
\omega_{2}(\mathscr{R}e(S^{\ast}(\phi_{-\alpha})))=
\frac{-(1+\alpha^{2})\cos t_{n}^{(n)}+2\alpha}{1-2\alpha \cos t_{n}^{(n)}+\alpha^{2}}.
\end{eqnarray*}
This ends the proof.
\end{proof}

\subsection{The first main result}\label{aa}
In view of this last result, is not surprising that there is a a connection the numerical radius of $S(\phi)$ and the eigenvalues of
the KMS matrix. We can now, under Propositions \ref{bb1} and \ref{akouda}, express the numerical radius of the truncated shift $S(\phi)$ where
$\phi$ is a finite Blaschke product with unique zero.
\begin{theorem}\label{Haykoulita}
 \textit{ Let $\phi(z)=\left( \dfrac{z-\alpha}{1-\overline{\alpha}z}\right)^n $ with $\alpha\in\mathbb{C}$ and $\lvert\alpha\lvert<1$.Then
$$\omega_{2}(S(\phi))=\frac{-(1+|\alpha|^{2})\cos t_{n}^{(n)}+2|\alpha|}{1-2|\alpha| \cos t_{n}^{(n)}+|\alpha|^{2}}.$$}
\end{theorem}
We conclude this section by giving some illustrating corollaries showing the importance of Theorem \ref{Haykoulita}.
\begin{corollary}
 \textit{ Let $\phi(z)=\left( \dfrac{z-\alpha}{1-\overline{\alpha}z}\right)^2 $ with $\alpha\in \mathbb{C}$ and $|\alpha|<1$. Then
$$\omega_{2}(S(\phi))=\dfrac{1+2|\alpha|-|\alpha|^{2}}{2}.$$}
\end{corollary}

\begin{proof}
We can assume that $0\leqslant\alpha<1$. This result is known, but it is interesting to notice that this result can also be obtained by using Theorem \ref{Haykoulita}.
This follows from the fact that
$p_{2}(\cos t)=4\cos^{2}t-4\alpha\cos t+\alpha^{2}-1$
and 
$\cos t_{2}^{(2)}=\dfrac{\alpha-1}{2}.$ 
\end{proof}

\begin{corollary}
 \textit{ Let $\phi(z)=\left( \dfrac{z-\alpha}{1-\overline{\alpha}z}\right)^3 $ with $\alpha\in \mathbb{C}$ and $|\alpha|<1$. Thus $$\omega_{2}(S(\phi))=\dfrac{7|\alpha|-|\alpha|^{3}+(1+|\alpha|^{2})(|\alpha|^{2}+8)^{\tfrac{1}{2}}}{4+2|\alpha|^{2}
+2|\alpha|(|\alpha|^{2}+8)^{\tfrac{1}{2}}}.$$}
\end{corollary}
 
\begin{proof}
For $0\leqslant\alpha<1$ observe that
$$p_{3}(\cos t)=\dfrac{2}{\sin t}\big( \sin(2t)-\alpha\sin t \big) 
\big( \cos(2t)-\alpha \cos t\big). $$
Under the Proposition \ref{Haykoul117}, we know that $\cos t_{3}^{(3)}$ is the unique solution of $\cos(2t)-\alpha \cos t=0$ on 
$]\dfrac{\pi}{2},\dfrac{3\pi}{4}[$. Thus a straightforward calculation shows that
$$\cos t_{3}^{(3)}=\dfrac{\alpha-(\alpha^{2}+8)^{\tfrac{1}{2}}}{4}$$ and the result follows. 
\end{proof}

\begin{corollary}
 \textit{ Let $\phi(z)=\left( \dfrac{z-\alpha}{1-\overline{\alpha}z}\right)^4 $ with $\alpha\in \mathbb{C}$ and $|\alpha|<1$. One has: 
$$\omega_{2}(S(\phi))=\dfrac{-|\alpha|^{3}+|\alpha|^{2}+7|\alpha|+1+
\left(1+|\alpha|^{2}\right)\left(|\alpha|^{2}+2|\alpha|+5\right)^{\tfrac{1}{2}}}{2|\alpha|^{2}+2|\alpha|+4+
2|\alpha|\left(|\alpha|^{2}+2|\alpha|+5\right)^{\tfrac{1}{2}}}.$$}
\end{corollary}
 
\begin{proof}
We assume that $\alpha$ is positive. We know that
 $t_{4}^{(4)}$
is the unique solution of $\alpha\sin\dfrac{3t}{2}=\sin\dfrac{5t}{2}$ on the interval $] \tfrac{3\pi}{5},\tfrac{4\pi}{5}]$.
Using the identities $\sin(3x)=3\sin x-4\sin^{3}x, \cos(3x)=-3\cos x+4\cos^{3}x,$and $\cos(2x)=2\cos^{2}x-1,\label{E888}$, 
we obtain
$$\cos(t_{4}^{(4)})=
\dfrac{\alpha+3-(\alpha^{2}+2\alpha+5)^{\tfrac{1}{2}}}{4}-1.$$
The desired equality follows immediately.
\end{proof}
 
\section{Application : A sharpened Schwarz-Pick operatorial inequality for nilpotent operator}\label{dd}
\hspace{0.5cm}For any $\rho>0$, we denote by $C_{\rho}(\mathcal{H})$ the set of all operator $T$ on $\mathscr{B}(\mathcal{H})$ which admit
 a unitary $\rho$-dilation in the sens of Nagy-Foias \cite{Nagy}, \cite{Nagy1}. This means that there exists a Hilbert space 
$\mathcal{K}\supseteq\mathcal{H}$ and an unitary operator $U$ acting on $\mathcal{K}$ such that 
$$T^{n}=\rho ~\mbox{pr}_{\mathcal{H}}(U^{n}).$$
\hspace{0.5cm}It is known that $C_{1}(\mathcal{H})$ consists of all contractions on $\mathcal{H}$, and
that $T\in C_{2}(\mathcal{H})$ if and only if the numerical range of $T$ is contained in the closed unit disc \cite{Berger}.

According to J. Holbrook \cite{Holbrook} and J. Williams \cite{Williams} we define the $\rho$-numerical radius of an operator $T$ in 
$\mathscr{B}(\mathcal{H})$ by the formula $$w_{\rho}(T)=\inf\{1/r : r>0 ~\mbox{et}~ rT\in\mathcal{C}_{\rho}\}.$$
Obviously, an operator $T$ belongs to $C_{\rho}(\mathcal{H})$ if and only if $w_{\rho}(T)\leq1$. Consequently, the operators in 
$C_{\rho}(\mathcal{H})$ are contractions with respect to the $\rho$-numerical radius, and according to this fact, any operator $T\in C_{\rho}(\mathcal{H})$
will be called a $\rho$-contraction on $\mathcal{H}$ (See for instance \cite{Cassier2}, \cite{Suciu2}, \cite{Suciu3}, \cite{Woederman}, \cite{Gaaya3} where we can find recent results about the $\rho$-numerical radius.).
Recall that $w_{1}(T)=\rVert T\rVert$ and $w_{2}(T)$ is the classical numerical radius of $T$. In \cite{Suciu1}, G. Cassier and
N. Suciu proved the following sharpened von Neumann inequality.
\begin{theorem}[\cite{Suciu1}]\label{roumain}
 \textit{Let $f$ be a non-constant analytic self map of the unit disc $\mathbb{D}$, and let  $\alpha\in\mathbb{D}$ and $m$ 
be the order of multiplicity of the zero $\alpha$ for the function $f-f(\alpha)$. Then for every operator $T\in\mathscr{B}(\mathcal{H})$ 
with $w_{\rho}(T)<1$ for some $\rho>0$, we have 
$$w_{\rho(\alpha)}\left[\Big(f(\alpha)I-f(T)\Big)\Big(I-\overline{f(\alpha)}f(T)\Big)^{-1}\right]
\leq \Bigg(w_{\rho(\alpha)}\left[\Big(\alpha I-T\Big)\Big(I-\overline{\alpha}T\Big)^{-1}\right]\Bigg)^{m},$$
where
$$
\rho(\alpha) = \left\{
    \begin{array}{ll}
        1+(\rho-1)\dfrac{1-|\alpha|}{1+|\alpha|} & \mbox{if } \rho\leq1 \\\\
        1+(\rho-1)\dfrac{1+|\alpha|}{1-|\alpha|} & \mbox{if } \rho\geq1
    \end{array}
\right.
.$$}
\end{theorem}
In particular, notice that for $\rho=1+\dfrac{1-|\alpha|}{1+|\alpha|} \geq1$ we have $\rho(\alpha)=2$. According to Theorem \ref{roumain}, 
it becomes easily to obtain the following Corollary
\begin{corollary}\label{haykaletta}
 \textit{ Let $\alpha\in\mathbb{D}$ and $T\in\mathcal{C}_{\rho}$ with $\rho=1+\dfrac{1-|\alpha|}{1+|\alpha|}$.
Let $f$ be a non-constant analytic self map of the unit disc $\mathbb{D}$ and $m$ be the order of multiplicity 
of the zero $\alpha$ for the function $f-f(\alpha)$. Then
$$w_{2}\left[\Big(f(\alpha)I-f(T)\Big)\Big(I-\overline{f(\alpha)}f(T)\Big)^{-1}\right]
\leq\Bigg( w_{2}\left[\Big(\alpha I-T\Big)\Big(I-\overline{\alpha}T\Big)^{-1}\right]\Bigg)^{m}.$$}
\end{corollary}
The next result, which is the main goal of this paper, give a Schwarz-Pick operatorial inequality for nilpotent operators.
\begin{theorem}\label{habibi}
 \textit{Let $T\in\mathscr{B}(\mathcal{H})$ be a nilpotent contraction satisfying $T^{n}=0$ and $\alpha\in\mathbb{D}$. Let
 $f\in\mathbb{A}(\mathbb{D})$ be a non-constant analytic self map of the unit disc $\mathbb{D}$ and $m$ be the order of multiplicity 
of the zero $\alpha$ for the function $f-f(\alpha)$. Then
$$w_{2}\left[\Big(f(\alpha)I-f(T)\Big)\Big(I-\overline{f(\alpha)}f(T)\Big)^{-1}\right]
\leq \left(\frac{-(1+|\alpha^{2}|)\cos t_{n}^{(n)}+2|\alpha|}{1-2|\alpha| \cos t_{n}^{(n)}+|\alpha|^{2}}\right)^{m}.$$}
\end{theorem}

\begin{proof}
Let $T\in\mathscr{B}(\mathcal{H})$ such that $T^{n}=0$ for some $n \geq 2$. Under Theorem \ref{catalin}, 
Corollary \ref{haykaletta} and Theorem \ref{nika} on obtain successively
\begin{eqnarray}
& &w_{2}\left[\Big(f(\alpha)I-f(T)\Big)\Big(I-\overline{f(\alpha)}f(T)\Big)^{-1}\right]\nonumber\\
&\leq& w_{2}\left[\Big(f(\alpha)I-f(S^{\ast}_{n})\Big)\Big(I-\overline{f(\alpha)}f(S^{\ast}_{n})\Big)^{-1}\right]\\
&\leq& w_{2}\left[\Big(\alpha I-S^{\ast}_{n}\Big)\Big(I-\overline{\alpha}S^{\ast}_{n}\Big)^{-1}\right]^{m}\\
&=&\big(w_{2}(S^{\ast}(\phi))\big)^{m}
\end{eqnarray}
with $ \phi(z)=\left(\dfrac{z-\alpha}{1-\overline{\alpha}z}\right)^{n}.$ This allows to establish the desired inequality.
\end{proof}
\begin{remark}
\textit{ As mentioned in the introduction, the Haagerup and 
de la Harpe Theorem \ref{Haa} is the special case $f(z)=z$ and $\alpha=0$ of Theorem \ref{habibi}.}
\end{remark}

\section{An estimate of the numerical radius of $S(\phi)$ where $\phi$ is a finite Blaschke product}\label{ee}
In this section we give an estimate of the numerical radius of $S(\phi)$ in the general case where $\phi$ is a finite Blaschke
product with different zeros. To state it, we need some results established by N. Nikolski and V. Vasyunin \cite{Nikolski}.
\begin{definition}
\textit{ Let $\varphi$ be an inner function and let $\mu_{\varphi}=\mu_{s}+\mu_{B}$, where $\mu_{s}$ is the singular measure
associated to the singular part of $\varphi$ and $\mu_{B}$ the measure defined by 
$$d\mu_{B}(\xi)=\frac{1}{2}\sum_{\lambda\in\varphi^{-1}(0)}k_{\varphi}(\xi)(1-|\xi|^{2})d\delta_{\lambda}(\xi).$$ Here 
$k_{\varphi}(\xi)$
denotes the order of multiplicity of the zero $\xi$ for the function $\varphi$ (with the understanding that $k_{\varphi}(\xi)=0$
if $\xi$ is not zero of $\varphi$). We say that $\mu_{\varphi}$ is the representing measure of $\varphi$.
}
\end{definition}  
 
\begin{theorem}[\cite{Nikolski}]\label{angle}
\textit{For $i=1,2$, let $\varphi_{i}$ be inner functions with the representing measure $\mu_{i}, i=1,2$. Then}
\begin{eqnarray*}
\sin\left<H(\varphi_{1}),H(\varphi_{2})\right>&\geq&\exp\left\{4
\int_{\overline{\mathbb{D}}}\int_{\overline{\mathbb{D}}}
\dfrac{\log\left|\dfrac{\zeta-\xi}{1-\overline{\zeta}\xi}\right|}
{(1-|\zeta|^{2})(1-|\xi|^{2})}d\mu_{1}(\zeta)d\mu_{2}(\xi)\right\}\\
&=&F(\varphi_{1},\varphi_{2}).
\end{eqnarray*}
\end{theorem}
 Note that in the special case where $\varphi_{i}$, i=1,2 are Blaschke products with unique zero defined by
 defined by $\phi_{i}=\left( \dfrac{z-\alpha_{i}}{1-\overline{\alpha_{i}}z}\right)^{n_{i}}$, we have
$$F(\varphi_{1},\varphi_{2})=\left|\dfrac{\alpha_{1}-\alpha_{2}}
{1-\overline{\alpha_{1}}\alpha_{2}}\right|^{2n_{1}n_{2}}.$$
In the sequel of this manuscript, we will denote by
 $$\phi_{i}=\left( \dfrac{z-\alpha_{i}}{1-\overline{\alpha_{i}}z}\right)^{n_{i}}$$
 for $1\leq i\leq p$ , $\delta=\max\{w_{2}(S^{\ast}(\phi_{i})), 1\leq i \leq p\}$ 
and $\rho=\max\{\cos\theta_{i,j}, 1\leq i<j \leq p\}$ where
 $\theta_{ij}$ denote the angle between the model subspaces $H(\phi_{i})$ and $H(\phi_{j})$.

\begin{theorem}\label{tunis}
\textit{ Let  $\phi=\prod_{i=1}^{p}\phi_{i}$ with $p\geq2$. If $\rho<\dfrac{1-\delta}{2(p-1)}$, then
$$ w_{2}(S(\phi))\leq\dfrac{\delta+\rho(p-1)}{1-\rho(p-1)}=G(\rho,\delta).$$}
\end{theorem}

In the beginning, we need the following lemmas which are easily verified. The proofs are left for 
the reader.

\begin{lemma}\label{haykoula}
 \textit{Let $(x_{n})$ a sequence of real numbers. Then for any integer $p\geq 2$ we have
$$\sum_{1\leq i < j \leq p}(x_{i}+x_{j})=(p-1)\sum_{1\leq k\leq p} x_{k}$$}
\end{lemma}
\begin{lemma}\label{haykoula ya hnine}
\textit{ Let $n\geq2$ and consider $B$ the $n\times n$ matrix defined by} $$B=
\left(
\begin{array}{ccccc}
0& 1&1                  & \dots&1   \\
   1         & 0 & 1&\cdots                    &       1\\
1                    &1                   & 0&\cdots               &    1    \\
   
\cdots&\cdots&\cdots&\cdots&  \\
1                   &       1            &      1 &\cdots            &0
\end{array}
\right)
.$$
\textit{ $B$ have two eigenvalues -1 and $n-1$. Here $n-1$ is a simple eigenvalue.}
\end{lemma}
\begin{proof}
Let $f\in H(\phi)$. There exists
$f_{1},f_{2},\dots,f_{p}$ belonging respectively in\\ $ H(\phi_{1}),H(\phi_{2}),\dots,H(\phi_{p})$ such that
$f=\sum_{i=1}^{p} f_{i}$. Therefore
\begin{eqnarray*}
 < S^{\ast}(\phi)f,f>&=&\Big(\sum_{k=1}^{p}\rVert f_{k}\rVert^{2}\Big)\Bigg\{\sum_{i=1}^{p}
\dfrac{\rVert f_{i}\rVert^{2}}{\sum_{k=1}^{p}\rVert f_{k}\rVert^{2}
}<\dfrac{S^{\ast}(\phi_{i})f_{i}}{\rVert f_{i}\rVert},\dfrac{f_{i}}{\rVert f_{i}\rVert}> \\
&+& \sum_{1\leq i\neq j \leq p}
\dfrac{\rVert S^{\ast}(\phi_{i})f_{i}\rVert\rVert f_{j}\rVert}{\sum_{k=1}^{p}\rVert f_{k}\rVert^{2}
}<\dfrac{S^{\ast}(\phi_{i})f_{i}}{\rVert S^{\ast}(\phi_{i})f_{i}\rVert},\dfrac{f_{j}}{\rVert f_{j}\rVert}>\Bigg\}.
\end{eqnarray*}
Which implies that 
\begin{eqnarray*}
|< S^{\ast}(\phi)f,f>|&\leq& \Big(\sum_{k=1}^{p}\rVert f_{k}\rVert^{2}\Big)\Bigg\{\sum_{i=1}^{p}
\dfrac{\rVert f_{i}\rVert^{2}}{\sum_{k=1}^{p}\rVert f_{k}\rVert^{2}
}w_{2}(S^{\ast}(\phi_{i})) \\
&+& \sum_{1\leq i\neq j \leq p}
\dfrac{\rVert f_{i}\rVert\rVert f_{j}\rVert}{\sum_{k=1}^{p}\rVert f_{k}\rVert^{2}
}\cos \theta_{ij}\Bigg\}.
\end{eqnarray*}
Hence
\begin{eqnarray*}
|< S^{\ast}(\phi)f,f>|&\leq& \Big(\sum_{k=1}^{p}\rVert f_{k}\rVert^{2}\Big)\Bigg\{\sum_{i=1}^{p}
\dfrac{\rVert f_{i}\rVert^{2}}{\sum_{k=1}^{p}\rVert f_{k}\rVert^{2}
}\delta\\
&+& \sum_{1\leq i\neq j \leq p}
\dfrac{\rVert f_{i}\rVert\rVert f_{j}\rVert}{\sum_{k=1}^{p}\rVert f_{k}\rVert^{2}
}\rho\Bigg\}.
\end{eqnarray*}
That is $$|< S^{\ast}(\phi)f,f>|\leq \Big(\sum_{k=1}^{p}\rVert f_{k}\rVert^{2}\Big)<AX,X>$$
with $$A=
\left(
\begin{array}{ccccc}
\delta& \rho&\rho                  & \dots&\rho   \\
   \rho         & \delta & \rho&\cdots                    &       \rho\\
\rho                    &\rho                   & \delta&\cdots               &    \rho    \\
   
\cdots&\cdots&\cdots&\cdots&  \\
\rho                   &       \rho             &      \rho &\cdots            &\delta
\end{array}
\right)
$$ and $X$ the unit vector defined by $$X=\dfrac{1}{\big(\sum_{k=1}^{p}\rVert f_{k}\rVert^{2}\big)^{\frac{1}{2}}}  
 \left( \begin{array}{c}
\rVert f_{1}\rVert \\
\rVert f_{2}\rVert\\
\vdots \\
\vdots \\
\rVert f_{p}\rVert
\end{array} \right).$$
Note that 
\begin{eqnarray}
 \sum_{k=1}^{p}\rVert f_{k}\rVert^{2}&=&1-\sum_{1\leq i\neq j \leq p}<f_{i},f_{j}>\nonumber\\
&\leq& 1+2\sum_{1\leq i<j \leq p}\rVert f_{i}\rVert\rVert f_{j}\rVert \cos\theta_{i,j}\nonumber\\
&\leq& 1+\rho\sum_{1\leq i<j \leq p}2\rVert f_{i}\rVert\rVert f_{j}\rVert,\nonumber\\
&\leq& 1+\rho\sum_{1\leq i<j \leq p}\big(\rVert f_{i}\rVert^{2}+\rVert f_{j}\rVert^{2}\big)\nonumber\\
&=&1+(p-1)\rho\sum_{k=1}^{p}\rVert f_{k}\rVert^{2}\label{E26}.
\end{eqnarray}
Here the equality (\ref{E26}) is due to Lemma \ref{haykoula}. Now, since $$\rho<\dfrac{1-\delta}{2(p-1)}<\dfrac{1}{p-1},$$
then $$\sum_{k=1}^{p}\rVert f_{k}\rVert^{2}\leq\dfrac{1}{1-\rho(p-1)}.$$ So, under Lemma \ref{haykoula ya hnine}
\begin{eqnarray*}
 w_{2}(S^{\ast}(\phi))&\leq&\dfrac{1}{1-\rho(p-1)}w_{2}(A)\\
&=&\dfrac{\delta+\rho(p-1)}{1-\rho(p-1)}<1.
\end{eqnarray*}
\end{proof}
\begin{remark}
 \textit{The estimate in Theorem \ref{tunis} is optimal when $\rho$ is small enough. In such case $G(\rho,\delta)$ tends to $\delta$
 which is completely natural because $S^{\ast}(\phi)$ tends  to the orthogonal sum of the $S^{\ast}(\phi_{i})$ for $1\leq i\leq p$.
Consequently the numerical range tends to become the convex hull of the numerical range of the $S^{\ast}(\phi_{i})$ for $1\leq i\leq p$.}
\end{remark}

We close this section by deriving an estimate of the numerical radius of $S(\phi)$ in the case where $\phi$ is finite Blaschke product with 
two different zeros.
\begin{corollary}
 \textit{ Let $\phi_{i}=\left( \dfrac{z-\alpha_{i}}{1-\overline{\alpha_{i}}z}\right)^{n_{i}}$ for $i=1,2$ with
$\alpha_{i}\in\mathbb{C}$ and $|\alpha_{i}|<1$. Let $\phi=\phi_{1}\phi_{2}$ and
 $\delta=\max\{w_{2}(S^{\ast}(\phi_{i})),i=1,2\}$. 
If $$\left(1-\left|\dfrac{\alpha_{1}-\alpha_{2}}{1-\overline{\alpha_{1}}\alpha_{2}}\right|^{2n_{1}n_{2}}\right)^{\frac{1}{2}}
<\dfrac{1-\delta}{2},$$ then
$$ w_{2}(S(\phi))\leq\dfrac{\delta+\left(1-\left|\dfrac{\alpha_{1}-\alpha_{2}}
{1-\overline{\alpha_{1}}\alpha_{2}}\right|^{2n_{1}n_{2}}\right)^{\frac{1}{2}}}{1-\left(1-\left|\dfrac{\alpha_{1}-\alpha_{2}}
{1-\overline{\alpha_{1}}\alpha_{2}}\right|^{2n_{1}n_{2}}\right)^{\frac{1}{2}}}.$$}
\end{corollary}
\begin{proof}
 The proof is an immediate consequence of the previous theorem and Theorem \ref{angle} of N. Nikolski and V. Vasyunin on angles 
between the model subspaces.
\end{proof}

\end{document}